\documentclass[11pt]{article}
\usepackage[english]{babel}
\usepackage[latin1]{inputenc}
\usepackage{latexsym, amsfonts}
\usepackage{amssymb, amsmath}
\usepackage{amsthm}
\usepackage[table]{xcolor}
\usepackage{eucal}
\def\dom{\mathop{\mathrm{Dom}}\nolimits}
\def\im{\mathop{\mathrm{Im}}\nolimits}
\def\ker{\mathop{\mathrm{Ker}}\nolimits} 

\def\fix{\mathop{\mathrm{Fix}}\nolimits}
\def\id{\mathrm{id}}

\newtheorem{theorem}{Theorem}
\newtheorem{lemma}[theorem]{Lemma}
\newtheorem{corollary}[theorem]{Corollary}
\newtheorem{proposition}[theorem]{Proposition}

\date{}

\begin{document}

\title{On the monoid of partial order-preserving transformations of a finite chain whose domains and ranges are intervals}

\author{Hayrullah Ay\i k, V\'\i tor H. Fernandes~and Emrah Korkmaz}

\maketitle

\begin{abstract} 
In this paper, we consider the monoid $\mathcal{PIO}_{n}$, of all partial order-preserving transformations on a chain with $n$ elements whose domains and ranges are intervals, along with its submonoid $\mathcal{PIO}_{n}^-$ of order-decreasing transformations. 
Our main aim is to give presentations for $\mathcal{PIO}_{n}^-$ and $\mathcal{PIO}_{n}$. 
Moreover, for both monoids, we describe regular elements and determine their ranks, cardinalities and the numbers of idempotents and nilpotents. 
\end{abstract}


\medskip 


\section{Introduction and preliminaries} 

Let $M$ be a monoid and let $A$ be a subset of $M$. The submonoid generated by $A$, i.e. the smallest submonoid of $M$ containing $A$, 
is denoted by $\langle A\, \rangle$, and if $M=\langle A \rangle$, then $A$ is called a \emph{generating set} of $M$. 
The \emph{rank} of $M$ is the minimum size of a generating set for $M$. An element $s\in M$ is said to be \emph{undecomposable}, if there are no $a,b\in M\setminus \{s\}$ such that $s=ab$. It is clear that every generating set of $M$ contains all undecomposable elements of $M$. Therefore, if a generating set $A$ of $M$ consists of only undecomposable elements of $M$, then $A$ is clearly the unique generating set of minimum size of $M$, 
and so the rank of  $M$ is $\lvert A\rvert$.
For comprehensive background information on semigroups, readers are directed to consult the book by Howie \cite{Howie:1995}.

\smallskip 

Let $n$ be a natural number and let $\Omega_{n}=\{1,\ldots ,n\}$. 
Denote by $\mathcal{PT}_{n}$ the monoid (under composition) of all partial transformations on $\Omega_{n}$ 
and by $\mathcal{T}_{n}$ its submonoid consisting of all full transformations on $\Omega_{n}$. 
Let $\alpha\in \mathcal{PT}_{n}$. We denote the domain and image (range) of $\alpha$ by $\dom(\alpha)$ and $\im(\alpha)$, respectively. 
The set of fixed points of $\alpha$ is denoted by $\fix(\alpha)$, i.e.  $\fix(\alpha) = \{ x\in \dom(\alpha )\mid x\alpha=x\}$, and the kernel of $\alpha$ by 
 $\ker(\alpha)$, i.e. $\ker(\alpha) =\{ (x,y)\in\dom(\alpha)\times\dom(\alpha) \mid x\alpha=y\alpha\}$. 
 The empty and identity transformations of $\mathcal{PT}_{n}$ are denoted by $0_{n}$ and $1_{n}$, respectively. 
 
From now on, let us consider $\Omega_{n}$ as a chain for the natural order. 
A transformation $\alpha \in \mathcal{PT}_{n}$ is called \emph{order-preserving} if $x\leqslant y$ implies $x\alpha\leqslant y\alpha$ for all $x,y \in \dom(\alpha)$. We denote by $\mathcal{PO}_{n}$ the submonoid of $\mathcal{PT}_{n}$ consisting of all order-preserving partial transformations, and by $\mathcal{O}_{n}$ the submonoid of $\mathcal{T}_{n}$ consisting of all order-preserving transformations. 
Semigroups of order-preserving transformations have attracted growing interest for many decades. In 1962 A\u\i zen\v stat \cite{Aizenstat:1962,Aizenstat1:1962} gave a presentation and a description of the congruences of  $\mathcal{O}_{n}$. 
In the same year, Popova \cite{Popova:1962}  gave a presentation for $\mathcal{PO}_{n}$. 
Over an extended period, significant attention has been devoted to studying algebraic and combinatorial properties of $\mathcal{O}_{n}$ and $\mathcal{PO}_{n}$ and many interesting results have emerged (see, for example, \cite{Fernandes:2001,  Fernandes&Santos:2019, Garba:1979,  Gomes&Howie:1992, Howie:1971, Laradji&Umar:2004, Laradji&Umar:2006,  Yang:2000}).  

\smallskip 

A subset $I$ of $\Omega_{n}$ is called an \emph{interval} if for all $x,y,z\in \Omega_{n}$, $x\leqslant y\leqslant z$ and $x,z\in I$ imply $y\in I$. 
Let $\mathcal{PIO}_{n}$ be the subset of $\mathcal{PO}_{n}$ consisting of all partial transformations $\alpha\in \mathcal{PO}_{n}$ such that the domain and image of $\alpha$ are both intervals of $\Omega_{n}$.  It is easy to show that $\mathcal{PIO}_{n}$ is a submonoid of $\mathcal{PO}_{n}$. 
Let us consider also the submonoid $\mathcal{IO}_{n} = \mathcal{PIO}_{n}\cap \mathcal{T}_{n}$ of $\mathcal{PIO}_{n}$ (and of $\mathcal{O}_{n}$). 

\smallskip 

Next, let us consider the submonoids 
$$
\mathcal{PT}_{n}^{-} = \{ \alpha\in \mathcal{PT}_{n} \mid \mbox{$x\alpha\leqslant x$, for all $x\in \dom(\alpha)$}\}
$$
and 
$$
\mathcal{PT}_{n}^{+} = \{ \alpha\in \mathcal{PT}_{n} \mid \mbox{$x\leqslant x\alpha $, for all $x\in \dom(\alpha)$}\}
$$
of all \textit{order-decreasing} and \textit{order-increasing} transformations of $\mathcal{PT}_{n}$, respectively. 
It is well known that $\mathcal{PT}_{n}^{-}$ and $\mathcal{PT}_{n}^{+}$ are isomorphic submonoids of $\mathcal{PT}_{n}$. 
In fact, let us consider the mapping $\phi:\mathcal{PT}_n\rightarrow \mathcal{PT}_n$ which maps each transformation $\alpha\in\mathcal{PT}_{n}$ into the transformation $\widetilde{\alpha}\in \mathcal{PT}_{n}$ defined by $\dom(\widetilde{\alpha}) =\{x\in\Omega_n \mid n+1-x\in\dom(\alpha)\}$ and $x\widetilde{\alpha}=n+1-(n+1-x)\alpha$ for $x\in\dom(\widetilde{\alpha})$. 
It is a routine matter to check that $\phi$ is an automorphism of monoids such that $\phi^2$ is the identity mapping of $\mathcal{PT}_{n}$, 
$\mathcal{PT}_{n}^{-}\phi \subseteq \mathcal{PT}_{n}^{+}$ and $\mathcal{PT}_{n}^{+}\phi \subseteq \mathcal{PT}_{n}^{-}$. 
Moreover, it is easy to check that we also have $\mathcal{PIO}_{n}\phi \subseteq \mathcal{PIO}_{n}$. 
Consequently, 
$$
\mathcal{PT}_{n}^{-}\phi=\mathcal{PT}_{n}^{+},\quad \mathcal{PT}_{n}^{+}\phi=\mathcal{PT}_{n}^{-} 
\quad \text{and} \quad \mathcal{PIO}_{n}\phi=\mathcal{PIO}_{n}.
$$
Now, consider the submonoids $\mathcal{PIO}_{n}^{-} =\mathcal{PIO}_{n} \cap \mathcal{PT}_{n}^{-}$ and $\mathcal{PIO}_{n}^{+}= \mathcal{PIO}_{n} \cap \mathcal{PT}_{n}^{+}$ of $\mathcal{PIO}_{n}$. It follows immediately from the previous properties that 
$$
\mathcal{PIO}_{n}^{-}\phi =\mathcal{PIO}_{n}^{+} \quad \text{and} \quad \mathcal{PIO}_{n}^{+}\phi =\mathcal{PIO}_{n}^{-},
$$
and so $\mathcal{PIO}_{n}^{-}$ and $\mathcal{PIO}_{n}^{+}$ are isomorphic submonoids of $\mathcal{PIO}_{n}$.

In \cite{Fernandes&Paulista:2023}, Fernandes and Paulista studied the monoid $\mathcal{IO}_n$. 
They showed that $\mathcal{IO}_{n}$ coincides with the monoid of all \emph{weak endomorphisms} of a a directed path with $n$ vertices. 
Furthermore, the authors studied the regularity, determined the rank, the cardinality and the number of idempotents of $\mathcal{IO}_{n}$. 
Thereafter, in \cite{Fernandes:2024}, Fernandes gave a presentation for the monoid $\mathcal{IO}_{n}$ and determined the cardinality, the rank and a presentation for the submonoid $\mathcal{IO}_{n}^{-}=\mathcal{IO}_{n}\cap\mathcal{PT}_{n}^-$ of $\mathcal{IO}_{n}$.

In this paper, we investigate the monoid  $\mathcal{PIO}_{n}$ and its submonoid $\mathcal{PIO}_{n}^{-}$, extending the studies aforementioned in the previous paragraph. Our primary contribution is to  provide a presentation for  $\mathcal{PIO}_{n}$. We also describe its regular elements, determine its rank, cardinality, and the numbers of idempotents and nilpotents.  Furthermore, we give a presentation for $\mathcal{PIO}_{n}^{-}$ and determine its rank (including the rank of its nilpotent subsemigroup), cardinality and the numbers of idempotents and nilpotents.

We would  like to mention the use of computational tools, namely GAP \cite{GAP4}. 

\medskip 

We end this section by recalling some notions related to the concept of a monoid presentation. 

For a set $A$, let $A^{*}$ denote the free monoid on $A$ consisting of all finite words over $A$. A \emph{monoid presentation} is an ordered pair $\langle A\mid R\rangle$ where $A$ is an alphabet and $R$ is a subset of $A^{*}\times A^{*}$. Each element $(u,v)$ of $R$ is called a \textit{(defining) relation}, and it is usually written by $u=v$. The empty word is denoted by $1$. For $n$ letters $a_{1}, \ldots, a_{n}$ of $A$, if $w=a_{1}\cdots a_{n}$, then $n$ is called the \emph{length} of $w$, denoted by $\lvert w\rvert =n$. (The length of the empty word is defined to be $0$.) A monoid $M$ is said to be \textit{defined by a presentation} $\langle A\mid R\rangle$ if $M$ is isomorphic to $A^{*}/\sim_{R}$, where $\sim_R$ denotes the congruence on $A^*$ generated by $R$, i.e. $\sim_{R}$ is the smallest congruence on $A^{*}$ containing $R$. Let  $X$ be a generating set of a monoid $M$ and let $\phi : A\rightarrow M$ be an injective mapping such that $A\phi =X$. If $\varphi : A^{*}\rightarrow M$ is the (unique) homomorphism that extends $\phi$ to $A^{*}$, then we say that $X$ \textit{satisfy} a relation $u=v$ of $A^{*}$  if $u\varphi =v\varphi$. For more details, see \cite{Lallement:1979} or \cite{Ruskuc:1995}.

A well-known direct method to obtain a presentation for a monoid is given by the following result; see, for example, \cite[Proposition 1.2.3]{Ruskuc:1995}.

\begin{proposition}\label{provingpresentation}
Let $M$ be a monoid generated by a set $X$ and let $A$ be an alphabet with the same size as $X$. 
Then, $\langle A\mid R\rangle$ is a presentation for $M$ if and only if the following two conditions are satisfied:
\begin{enumerate}
\item The generating set $X$ of $M$ satisfies all the relations from $R$;
\item If $w,w'\in A^*$ are any two words such that the generating set $X$ of $M$ satisfies the relation $w=w'$, then $w\sim_R w'$.
\end{enumerate}
\end{proposition}

For a finite monoid, a usual method to find a presentation is described by the following result 
(adapted to the monoid case from \cite[Proposition 3.2.2] {Ruskuc:1995}). 

\begin{proposition}[Guess and Prove method] \label{ruskuc} 
Let $M$ be a finite monoid generated by a set $X$ and let $A$ be an alphabet with the same size as $X$. 
Let $R\subseteq A^{*}\times A^{*}$ and $W\subseteq A^{*}$. Suppose that the following conditions are satisfied:
\begin{enumerate}
\item The generating set $X$ of $M$ satisfies all relations from $R$;
\item For each word $w\in A^{*}$, there exists a word $w'\in W$ such that $w\sim_R w'$;
\item $|W|\leqslant |M|$.
\end{enumerate}
Then, $M$ is defined by the presentation $\langle A\mid R\rangle$. 
\end{proposition}

Notice that, if $W$ satisfies the above conditions then, in fact, $\lvert W\rvert =\lvert M\rvert$. Moreover, we say that $W\subseteq A^{*}$ is a set of  \emph{canonical forms} for a finite monoid $M$ with respect to $R\subseteq A^{*} \times A^{*}$, if the conditions of Proposition \ref{ruskuc} are satisfied. 

\section{Some properties of $\mathcal{PIO}_n$} 

Let $Y$ be a non-empty subset of $\Omega_{n}$ and let $(A_{1}, \ldots ,A_{r})$ be a partition of $Y$. 
If $x<y$ for all $x\in A_{i}$ and $y\in A_{i+1}$, for all $1\leqslant i\leqslant r-1$, and if each $A_{i}$ ($1\leqslant i\leqslant r$) is an interval of $\Omega_{n}$, then $(A_{1}, \ldots ,A_{r})$ is called an \emph{ordered partition into intervals} of $Y$. 
For any non-negative integers $r$ and $s$, let 
$$ 
[r,s] = \left\{ 
	\begin{array}{cl}
		\emptyset & \mbox{ if } s<r \vspace*{1mm} \\
		\{ i\in \mathbb{N} \mid r\leqslant i\leqslant s\} & \mbox{ if } r\leqslant s,
	\end{array} \right. 
$$
and, for $0\leqslant r, s\leqslant n$, let 
$$
J_{r}=\{ \alpha\in \mathcal{PIO}_{n} \mid \lvert \im(\alpha) \rvert=r \}\quad\text{and}\quad K_{s}=\{ \alpha\in \mathcal{PIO}_{n} \mid \lvert \dom(\alpha) \rvert=s \}.
$$
Then, it is clear that $J_{0}=K_{0} =\{ 0_{n} \}$ and $J_{n}=K_{n} =\{ 1_{n} \}$. 
Moreover, for $\alpha \in J_{r}\cap K_{s}$, with $1\leqslant r \leqslant s\leqslant n$, there exist $0\leqslant k\leqslant n-r$ 
and an ordered partition into intervals $(A_{1}, A_{2}, \ldots, A_{r})$ of $\dom(\alpha)$ such that 
$$
\alpha =\left(\begin{array}{cccc}
	A_{1} & A_{2} & \cdots & A_{r}\\
	k+1  &  k+2  & \cdots & k+r
\end{array}\right) .
$$

Next, recall the following well known identities: 
\begin{equation}\label{e1}
	\begin{array}{lcl}
		\sum\limits_{k=0}^{n} \binom{n}{k}=2^{n}, & & \sum \limits_{k=1}^{n} k2^{k}=(n-1)2^{n+1}+2, \vspace*{2mm}\\
		\sum\limits_{k=1}^{n} k\binom{n}{k}=n2^{n-1}, &&
		\sum \limits_{k=1}^{n} k^{2}2^{k}=(n^{2}-2n+3)2^{n+1}-6. 
	\end{array}
\end{equation}

We can now establish the cardinality of $\mathcal{PIO}_{n}$. 

\begin{theorem}\label{3} 
	$\lvert \mathcal{PIO}_{n}\rvert=(n+3)2^n -n^2-3n-2$.
\end{theorem}

\begin{proof} For $1\leqslant r\leqslant s \leqslant n$, there exist $n-r+1$ intervals of $\Omega_{n}$ with $r$ elements,  $n-s+1$ intervals of $\Omega_{n}$ with $s$ elements and, if $Y$ is an interval of $\Omega_{n}$ with $s$ elements, $\binom{s-1}{r-1}$ ordered partitions into intervals of $Y$ with $r$ parts. Hence, for $1\leqslant s\leqslant n$, it follows from (\ref{e1}) that
	\begin{eqnarray*}
		\lvert K_{s}\rvert &=&(n-s+1)\sum\limits_{r=1}^{s}(n-r+1)\binom{s-1}{r-1}\\
		&=&(n-s+1)\left[(n+1)\sum\limits_{r=1}^{s}\binom{s-1}{r-1}- \sum\limits_{r=1}^{s}r\binom{s-1}{r-1}\right]\\
		&=&\left((2n^{2}+3n+1)-(3n+2)s+s^{2}\right)2^{s-2}.
	\end{eqnarray*}
In particular, $\lvert K_{1}\rvert=n^{2}$,  and so we obtain	
\begin{eqnarray*}
	\lvert \mathcal{PIO}_{n}\rvert &=&\sum\limits_{s=0}^{n}\lvert K_{s}\rvert \\ 
	&=& 1+n^{2}+ (2n^{2}+3n+1)\sum\limits_{s=2}^{n}2^{s-2} - (3n+2)\sum\limits_{s=2}^{n}s2^{s-2} +\sum\limits_{s=2}^{n}s^{2}2^{s-2}\\
	&=&(n+3)2^{n} -n^{2}-3n-2,
\end{eqnarray*}
as required.
\end{proof}

Let $S$ be a semigroup with zero $0$. 
An element $s\in S$ is called \textit{nilpotent} if there exists a positive integer $k$ such that $s^k=0$. 
Let us denote by $N(S)$ the set of all nilpotent elements of $S$. The semigroup $S$ is called \textit{nilpotent} if all its elements are nilpotent, 
i.e. if $N(S)=S$. Observe that, in general, $N(S)$ may not be a subsemigroup of $S$. 

Let  $\alpha\in \mathcal{PT}_{n}$. 
It is clear that, if $\alpha\in \mathcal{PT}_{n}$ is nilpotent, then $\fix(\alpha)=\emptyset$. 
On the other hand, it is easy to check that the converse is true for an \textit{aperiodic} element $\alpha\in\mathcal{PT}_{n}$
(i.e. such that there exists a positive integer $k$ with $\alpha^{k+1}=\alpha^k$), 
i.e. if $\alpha\in \mathcal{PT}_{n}$ is aperiodic and $\fix(\alpha)=\emptyset$, then $\alpha$ is nilpotent. 
In particular, since it is well known that all elements of $\mathcal{PO}_{n}$ are aperiodic, 
an element $\alpha\in \mathcal{PO}_{n}$ is nilpotent if and only if $\fix(\alpha)=\emptyset$ (see \cite[Lemma 2.1]{Garba:1994}). 
We now proceed to determine the cardinality of $N(\mathcal{PIO}_n)$. 

\begin{proposition}\label{4} 
	$\lvert N(\mathcal{PIO}_n)\rvert=2^{n+2}-n^2-3n-3$.
\end{proposition}

\begin{proof} Let $\alpha\in \mathcal{PIO}_{n}\setminus N(\mathcal{PIO}_{n})$.  Then, 
there exist $1\leqslant p\leqslant k\leqslant q\leqslant n$ such that $\dom(\alpha)=[p,q]$ and $k\alpha =k$. 
Since $(x-1)\alpha\in\{x\alpha,x\alpha-1\}$ for all $x\in\dom(\alpha)\setminus\{p\}$, we have $[p\alpha, k]\subseteq [p,k]$ and $[k, q\alpha]\subseteq [k,q]$, whence $\im(\alpha)\subseteq \dom(\alpha)$. Conversely, let $\alpha\in \mathcal{PIO}_{n}$ be such that $\emptyset \neq \im(\alpha) \subseteq \dom(\alpha)$.  Then, $x\alpha^{m}\in \dom(\alpha)$ for all $x\in \dom(\alpha)$ and $m\in \mathbb{N}$, from which follows that 
$\alpha \notin N(\mathcal{PIO}_{n})$. Thus, $\alpha\in \mathcal{PIO}_{n}\setminus\{0_n\}$ is non-nilpotent if and only if $\im(\alpha)\subseteq \dom(\alpha)$. 
	
For an interval $Y$ of $\Omega_{n}$ with $s$ elements, let 
$$
\hat{K}_{s}(Y) = \{ \alpha\in \mathcal{PIO}_{n} \setminus N(\mathcal{PIO}_{n}) \mid \dom(\alpha) =Y\} \quad\text{and}\quad  
	\hat{K}_{s} = K_{s}\cap (\mathcal{PIO}_{n} \setminus N(\mathcal{PIO}_{n})). 
$$
For any $\alpha \in \hat{K}_{s}(Y)$, since $\im(\alpha) \subseteq \dom(\alpha)$, it follows from \cite[Theorem 2.6]{Fernandes&Paulista:2023} (see also \cite{Fernandes:2024}) that $\lvert \hat{K}_{s}(Y) \rvert = \lvert \mathcal{IO}_{s} \rvert =(s+1)2^{s-2}$. Since there exist $n-s+1$ intervals of $\Omega_{n}$ with $s$ elements, it follows that $\lvert \hat{K}_{s}\rvert =(n-s+1) (s+1)2^{s-2}$ and so,  by (\ref{e1}), 
	\begin{eqnarray*}
		\lvert  \mathcal{PIO}_{n}\setminus N(\mathcal{PIO}_{n})\rvert &=& 
		\sum\limits_{s=1}^{n}(n-s+1)(s+1)2^{s-2}\\
		&=&n+(n+1)\sum\limits_{s=2}^{n}2^{s-2}+n\sum\limits_{s=2}^{n}s2^{s-2} - \sum\limits_{s=2}^{n}s^{2}2^{s-2}\\
		&=&(n-1)2^{n}+1.
	\end{eqnarray*}
Hence, by Theorem \ref{3}, we obtain $\lvert N(\mathcal{PIO}_{n})\rvert = 2^{n+2} -n^{2} -3n-3$.
\end{proof}

Recall that an element $\alpha\in \mathcal{PT}_{n}$ is idempotent  (i.e. $\alpha^2=\alpha$) if and only if $\im(\alpha)=\fix(\alpha)$.  
For a subset $U$ of $ \mathcal{PT}_{n}$, denote the set of all idempotents in $U$ by $E(U)$. 
Observe that neither $N(\mathcal{PIO}_{n})$ nor $E(\mathcal{PIO}_{n})$ are semigroups, 
and that $N(\mathcal{PIO}_{n}) \cap E(\mathcal{PIO}_{n}) =\{ 0_{n}\}$. 
It has been shown in \cite[Corollary 2.9]{Fernandes&Paulista:2023} that $\lvert E(\mathcal{IO}_{n})\rvert =\frac{n(n+1)}{2}$. 
Next, we determine the cardinality of $E(\mathcal{PIO}_{n})$. 

\begin{proposition}\label{5}
	$\lvert E(\mathcal{PIO}_n)\rvert=1+\frac{n(n+1)(n+2)(n+3)}{24}$.
\end{proposition}

\begin{proof}  For $1\leqslant s\leqslant n$ and an interval $Y$ of $\Omega_{n}$ with $s$ elements, it is clear that 
$\lvert E(\hat{K}_{s}(Y))\rvert =\lvert E(\mathcal{IO}_{s}) \rvert = \frac{s(s+1)}{2}$ and so 
$\lvert E(\hat{K}_{s}) \rvert =(n-s+1) \lvert E(\mathcal{IO}_{s})\rvert =(n-s+1) \frac{s(s+1)}{2}$, 
since there are $n-s+1$ possible domais. 
Therefore, by including the zero element $0_{n}$, we obtain 
\begin{eqnarray*}
	\lvert E(\mathcal{PIO}_n)\rvert &=& 1+\sum\limits_{s=1}^{n} (n-s+1)\frac{s(s+1)}{2}\\
	&=&1+\frac{n+1}{2}\sum\limits_{s=1}^{n}s +\frac{n}{2}\sum \limits_{s=1}^{n} s^{2} -\frac{1}{2}\sum\limits_{s=1}^{n} s^{3} \\
	&=&1+\frac{n(n+1)(n+2)(n+3)}{24},
\end{eqnarray*}
as expected.
\end{proof}

\medskip 

Recall that an element $s$ of a semigroup $S$ is called \emph{regular} if there exists $t\in S$ such that $s=sts$. 
A semigroup is \emph{regular} if all its elements are regular. We now provide a description of the regular elements within $\mathcal{PIO}_{n}$.

\smallskip 

For any $\alpha\in\mathcal{PIO}_{n}\setminus\{0_n\}$, we define a transformation $\overline{\alpha}\in\mathcal{T}_n$  by
\begin{eqnarray}\label{e2}
	x\overline{\alpha} &=& \left\{ \begin{array}{cl} 
		x\alpha & \mbox{if $x\in\dom(\alpha)$} \\
		\min(\im(\alpha)) & \mbox{if $x<\min(\dom(\alpha))$} \\
		\max(\im(\alpha))  &\mbox{if $ x>\max(\dom(\alpha))$}
	\end{array}\right.
\end{eqnarray}
for all $x\in \Omega_{n}$. Then, clearly, $\overline{\alpha}\in \mathcal{IO}_{n}$ and $\im(\overline {\alpha})= \im(\alpha)$.

\begin{proposition}\label{6}
	For any $\alpha\in\mathcal{PIO}_{n}$, the following properties are equivalent: 
	\begin{enumerate} 
		\item [$(i)$] $\alpha$  is regular;
		\item [$(ii)$] $\lvert x\alpha^{-1}\rvert>1$ implies  $x\alpha^{-1}\cap \{\min(\dom(\alpha)),\max(\dom(\alpha))\}\neq \emptyset$ for all $x\in \im(\alpha)$; 
		\item [$(iii)$] $\lvert x\alpha^{-1}\rvert>1$ implies $x\in \{\min(\im(\alpha)), \max(\im(\alpha))\}$ for all $x\in \im(\alpha)$.
	\end{enumerate}
\end{proposition}

\begin{proof}  First, notice that $\alpha\in \mathcal{PIO}_n$ satisfies property $(ii)$ if and only if $\overline{\alpha}$ satisfies property $(2)$ of \cite[Proposition 2.7]{Fernandes&Paulista:2023}, namely  $\lvert x\alpha^{-1} \rvert>1$ implies  $x\alpha^{-1} \cap \{1,n\}\neq \emptyset$ for all $x\in \im(\alpha)$. Moreover, for any $\alpha\in \mathcal{PIO}_n$ and $x\in \im(\alpha)$, it is clear that $\min(\dom(\alpha)) \in x\alpha^{-1}$ if and only if  $x=\min( \im(\alpha))$, and that $\max(\dom(\alpha))\in x\alpha^{-1}$ if and only if $x=\max( \im(\alpha))$. Therefore, properties $(ii)$ and $(iii)$ are equivalent, and so it is enough to show that properties $(i)$ and $(ii)$ are equivalent. 
Observe that, if $\alpha=0_n$, then $\alpha$ is regular and satisfies trivially $(ii)$. 
	
$(i)\Rightarrow (ii)$ Suppose that $\alpha \in \mathcal{PIO}_{n}\setminus\{0_n\}$ is regular. Then, there exists $\beta \in \mathcal{PIO}_n$ such that $\alpha =\alpha \beta \alpha$. Hence $\im(\alpha) \subseteq \dom(\beta)$ and $\im(\alpha\beta)\subseteq \dom(\alpha)$, 
 and so it is easy to check that $\overline{\alpha} =\overline{\alpha} \overline{\beta} \overline{\alpha}$. 
 Therefore, $\overline{\alpha}$ is a regular element of $\mathcal{IO}_{n}$ and so $\overline{\alpha}$ satisfies property $(2)$ of 
\cite[Proposition 2.7]{Fernandes&Paulista:2023}, which implies that $\alpha$ satisfies property $(ii)$.
	
$(ii)\Rightarrow (i)$ Now, suppose that $\alpha \in \mathcal{PIO}_{n}\setminus\{0_n\}$ satisfies property $(ii)$. 
Take $1\leqslant p\leqslant n$, $1\leqslant q\leqslant n$, $0\leqslant r\leqslant n-q$ and $0\leqslant s\leqslant n-p$ such that 
$\dom(\alpha)=[p,p+s]$ and $\im(\alpha)=[q,q+r]$. Then, there exists a unique $p\leqslant k\leqslant p+s $ such that 
the restriction $\alpha|_{[k,k+r]}$ of $\alpha$ to $[k,k+r]$ is injective and $[k,k+r]\alpha=\im(\alpha)$. 
Therefore, it is clear that $\beta=(\alpha|_{[k,k+r]})^{-1}\in \mathcal{PIO}_{n}$ and $\alpha=\alpha\beta\alpha$. 
Thus, $\alpha$ is regular, as required. 
\end{proof}

It is a routine matter to check that, for $n\leqslant 3$, all elements of  $\mathcal{PIO}_n$ satisfy the property $(2)$.
On the other hand, for $n\geqslant 4$,
\begin{eqnarray*}
	\left(\begin{array}{ccccccc}
		1 & 2 & 3 & 4 & 5 & \cdots & n \\
		1 & 2 & 3 & 3 & 4 & \cdots & n-1 
	\end{array}\right),
\end{eqnarray*}
is an element of $\mathcal{PIO}_{n}$ which does not satisfy the property $2$ (or $(3)$) above. Therefore, we have the following corollary.

\begin{corollary}\label{7}
	The monoid   $\mathcal{PIO}_{n}$ is regular if and only if $n\leqslant 3$.
\end{corollary}

\medskip 

We end this section by determining generators and the rank of  $\mathcal{PIO}_n$. 

\smallskip 

For $n\geqslant 2$ and $1\leqslant i\leqslant n-1$, let us consider the following elements of $\mathcal{PIO}_n$: 
\begin{equation}\label{gen} 
\begin{array}{c}
	a_i = \begin{pmatrix} 
		1 & \cdots & i & i+1 & \cdots   & n \\
		1 & \cdots & i & i   & \cdots   & n-1
	\end{pmatrix},  \quad  
	b_i = \begin{pmatrix} 
		1 & \cdots & i   & i+1 & \cdots   & n \\
		2 & \cdots & i+1 & i+1 & \cdots   & n
	\end{pmatrix}, \vspace*{3mm}\\ 
 e_{i}=\begin{pmatrix} 
	1 & \cdots & i  \\
	1 & \cdots & i 
\end{pmatrix} \quad\text{ and }\quad
f_{i+1}=\begin{pmatrix} 
	i+1 & \cdots   & n \\
	i+1 & \cdots   & n
\end{pmatrix}.
\end{array}
\end{equation}

Recall that, it has been shown in \cite[Proposition 3.3 and Theorem 3.5] {Fernandes&Paulista:2023} that, for $n\geqslant3$,  
$\{a_1,\ldots, a_{n-2},b_{n-1}\}$ is a generating set of minimum size of $\mathcal{IO}_{n}$. 
Let 
$$
\mathcal{U} =\{a_1, \ldots, a_{n-2}, b_{n-1}, e_{n-1}, f_{2}\}
$$ 
and, for each $1\leqslant r\leqslant n-1$ and $1\leqslant j\leqslant n-r+1$, let 
\begin{eqnarray*}
	\lambda_{j,r} &=& \left(\begin{array}{cccccc}
		j & j+1 & \cdots & j+r-1 \\
		j & j+1 & \cdots & j+r-1
	\end{array}\right). 
\end{eqnarray*} 

\begin{lemma}\label{8}
	For $1\leqslant r\leqslant n-1$ and $1\leqslant j\leqslant n-r+1$,  $\lambda_{j,r}\in \langle \mathcal{U} \rangle$. 
\end{lemma}

\begin{proof} For $r=n-1$, it follows that $1\leqslant j\leqslant 2$, 
$\lambda_{1,n-1}=e_{n-1}$ and $\lambda_{2,n-1}=f_{2}$, whence $\lambda_{j,r}\in \mathcal{U}\subseteq \langle \mathcal{U} \rangle$.  
Then, suppose that $1\leqslant r\leqslant n-2$. Let 
$$
	\alpha_{r} = \left( \begin{array}{ccc|ccccc}
		1     & \cdots & r+1 & r+2 & \cdots & n \\
		n-r-1 & \cdots & n-1 & n   & \cdots & n
		\end{array}\right) 
$$
and
$$
	\beta_{r} = \left( \begin{array}{ccc|cccccc}
			1 & \cdots & n-r-1 & n-r & \cdots & n \\
			1 & \cdots & 1     & 2   & \cdots & r+2
		\end{array}\right).
$$
Clearly, 
$\alpha_{r}, \beta_{r} \in \mathcal{IO}_{n} =\langle \{a_1,\ldots, a_{n-2}, b_{n-1}\} \rangle \subseteq \langle \mathcal{U}\rangle$
and $\alpha_{r}e_{n-1}\beta_{r}f_{2}=\left(\begin{smallmatrix}
	2 & \cdots & r+1 \\
	2 & \cdots & r+1
\end{smallmatrix}\right)$. 
On the other hand, for $1\leqslant j\leqslant n-r+1$, let 
$$
	\gamma_{j} = \left( \begin{array}{ccc|ccc|ccc}
		1 & \cdots & j-1 & j & \cdots & j+r-1 & j+r & \cdots & n \\
		1 & \cdots & 1   & 2 & \cdots &  r+1  & r+2 & \cdots & r+2
	\end{array}\right) 
$$
and
$$
	\delta_{j} = \left( \begin{array}{c|ccc|ccc}
		1 & 2 &\cdots & r+1   & r+2   & \cdots & n \\
		j & j &\cdots & j+r-1 & j+r-1 & \cdots & j+r-1
	\end{array}\right). 
$$
Hence, it is also clear that $\gamma_{j}, \delta_{j}\in \mathcal{IO}_{n} \subseteq\langle \mathcal{U} \rangle$ and 
$\lambda_{j,r} =\gamma_{j} \alpha_{r} e_{n-1} \beta_{r} f_{2} \delta_{j}$. 
Thus, $\lambda_{j,r} \in\langle \mathcal{U} \rangle$, as required.
\end{proof}

Now, we can state and prove the following result.

\begin{proposition}\label{9}
For $n\geqslant 3$,	$\mathcal{PIO}_{n}=\langle \mathcal{U} \rangle$.
\end{proposition}

\begin{proof} Let $\alpha\in \mathcal{PIO}_{n}$. Let $0\leqslant s\leqslant n$ be such that $\alpha\in K_{s}$. 
If $s=n$, then $\alpha\in \mathcal{IO}_{n} \subseteq\langle \mathcal{U} \rangle$. 
On the other hand, if $s=0$, then $\alpha =0_{n} = \lambda_{1,1} \lambda_{2,2} \in \langle \mathcal{U}\rangle$, by Lemma \ref{8}. 
So, suppose that $1\leqslant s\leqslant n-1$ and let $j=\min( \dom(\alpha))$.  Then, 
we have $\alpha=\lambda_{j,s}\overline{\alpha}$ and so, since $\overline{\alpha}\in \mathcal{IO}_{n} \subseteq\langle \mathcal{U} \rangle$ and by Lemma \ref{8}, we obtain $\alpha\in \langle \mathcal{U}\rangle$. 
\end{proof}

The following lemma can be proved in exactly the same way as \cite[Lemma 3.4]{Fernandes&Paulista:2023}.

\begin{lemma}\label{10}
Let $M$ be a submonoid of  $\mathcal{PT}_n$ with trivial group of units. 
Let $X$ be a set of generators of $M$. Then, $X$ contains at least one element with each of the possible kernels of elements with images of size $n-1$. 
\end{lemma}

Now, observe that $\mathcal{PIO}_{1}=\{0_{1}, 1_{1}\}$ and so, clearly, the rank of $\mathcal{PIO}_{1}$ is $1$. 
Also, it is routine matter to check that   
$\mathcal{U}=\left\{\left( \begin{smallmatrix}
	1 \\
	2 
\end{smallmatrix}\right),  \left( \begin{smallmatrix}
	2 \\
	1
\end{smallmatrix}\right) , \left( \begin{smallmatrix}
	1&2 \\
	1&1
\end{smallmatrix}\right)\right\}$ is a set of generators of $\mathcal{PIO}_{2}$ with minimum size, whence the rank of  $\mathcal{PIO}_{2}$ is $3$. 
On the other hand, for $n\geqslant3$, since $\mathcal{U}$ consists of $n+1$ elements with images of size $n-1$ all of them with different kernels, 
we immediately have the following result:

\begin{theorem}\label{11}
For $n\geqslant 2$, the rank of the monoid $\mathcal{PIO}_n$ is $n+1$. 
\end{theorem}

\section{Some properties of $\mathcal{PIO}_{n}^{-}$}

In this section, similar to the previous one, 
we investigate some combinatorial and algebraic properties of $\mathcal{PIO}_{n}^{-}$ and, since $\mathcal{PIO}_{n}^{+}$ and $\mathcal{PIO}_{n}^{-}$ are isomorphic, of $\mathcal{PIO}_{n}^{+}$ as well. 

For $0\leqslant r\leqslant n$, let us define $\mathcal{W}_r=\{\alpha\in \mathcal{PIO}_{n}^{-} \mid \lvert \dom\alpha \rvert=r\}$. 
We begin by calculating the cardinality of $\mathcal{PIO}_{n}^{-}$. 

\begin{theorem}\label{12} 
	$\lvert \mathcal{PIO}_{n}^{-}\rvert=2^{n+2}-\frac{(n+2)(n+3)}{2}$.
\end{theorem}

\begin{proof}  
First, observe that $\mathcal{W}_{0} =\{0_{n}\}$. 
Then, take $1\leqslant r\leqslant n$ and noticed that , for each $\alpha\in \mathcal{W}_{r}$, 
there exists $1\leqslant k\leqslant n-r+1$ such that $\dom(\alpha)=[k, k+r-1]$, and so, 
by \cite[Corollary 3.2] {Fernandes:2024}, it is easy to conclude that there are 
$k2^{r-1}$ elements in $\mathcal{W}_{r}$ with domain $[k, k+r-1]$. Hence, we get
$$
\lvert \mathcal{W}_{r}\rvert= \sum\limits_{k=1}^{n-r+1}k2^{r-1}=2^{r-2}(n^{2}+3n+2 -(2n+3)r+r^{2}).
$$ 
Since the sets $\mathcal{W}_{0},\dots, \mathcal{W}_{n}$ are disjoint, it follows that
\begin{eqnarray*}
	\lvert \mathcal{PIO}_{n}^{-}\rvert &=& 1+\sum\limits_{r=1}^{n}2^{r-2}(n^{2}+3n+2 -(2n+3)r+r^{2}) \\
	& = & 1 + (n^{2}+3n+2)\sum_{r=1}^{n}2^{r-2} - (2n+3) \sum_{r=1}^{n}2^{r-1} +\sum_{r=1}^{n}2^r\\ 
	&=&2^{n+2}-\frac{(n+2)(n+3)}{2},
\end{eqnarray*}
as claimed. 
\end{proof}

Next, we determine the number of idempotents in $\mathcal{PIO}_n^-$. 

\begin{proposition}\label{13}
	$\lvert E(\mathcal{PIO}_n^-)\rvert=1+\frac{n(n+1)(n+2)}{6}$.
\end{proposition}

\begin{proof} 
Let $\alpha\in  E(\mathcal{W}_{r})$, with $1\leqslant r\leqslant n$.  Then, $\dom(\alpha)= [i, i+r-1]$, for some $1\leqslant i\leqslant n-r+1$. 
Hence,  as $\fix(\alpha)=\im(\alpha)$, $\alpha$ must take the following form:  
\begin{eqnarray*}
	\alpha=	\left(\begin{array}{ccccccc}
		i & i+1  & \cdots  & i+k & i+k+1 & \cdots & i+r-1 \\
		i & i+1  & \cdots  & i+k & i+k   & \cdots & i+k 
	\end{array}\right),
\end{eqnarray*}
for some  $0\leqslant k\leqslant r-1$. 
Since the number of possibilities for $\dom(\alpha)$ is $n-r+1$, we obtain $\lvert E(\mathcal{W}_{r}) \rvert= r(n-r+1)$. 
By including the empty transformation $0_{n}$, the number of idempotents in $\mathcal{PIO}_{n}^{-}$ is then 
$$
	\lvert E(\mathcal{PIO}_n^-)\rvert= 1+\sum\limits_{r=1}^{n} r(n-r+1)=1+\frac{n(n+1)(n+2)}{6}, 
$$
as required. 
\end{proof} 

\medskip 

Now, let  $M$ be a submonoid of $\mathcal{PT}_n^-$. 
Then, it is easy to show that an element $\alpha\in M$ is regular if and only if $\alpha$ is idempotent. 
Moreover, it is also clear that the product of two nilpotent elements of $M$ is a nilpotent element of $M$, whence 
$N(M)$ is a (nilpotent) subsemigroup of $M$ or the empty set. 
Therefore, in particular, the only regular elements of $\mathcal{PIO}_n^-$ are its idempotents and $N(\mathcal{PIO}_n^-)$ is a subsemigroup of 
$\mathcal{PIO}_n^-$. 

\medskip 

Next, we compute the size and (semigroup) rank of $N(\mathcal{PIO}_{n}^{-})$. Observe that, $\mathcal{PIO}_{1}^{-}=\mathcal{PIO}_{1}=\{0_1,1_1\}$, 
whence $N(\mathcal{PIO}_{1}^{-})=\{0_1\}$, which clearly has rank $1$. 

\begin{proposition}\label{16} 
	$\lvert N(\mathcal{PIO}_{n}^{-})\rvert=2^{n+1}-\frac{(n+1)(n+2)}{2}$.
\end{proposition}

\begin{proof}  
For $n=1$, we have $\lvert N(\mathcal{PIO}_{n}^{-})\rvert=1$ and so the equality holds. So, suppose that $n\geqslant2$. 
For each $\alpha\in \mathcal{PIO}_{n-1}^{-}$, define $\widehat{\alpha}\in\mathcal{PT}_{n}$ by 
$\dom(\widehat{\alpha}) =\{ x\in \Omega_{n} \mid x-1\in \dom(\alpha)\}$, and $x\widehat{\alpha} =(x-1)\alpha$ for $x\in \dom(\widehat{\alpha})$. 
It is easy to check that $\widehat{\alpha}\in N(\mathcal{PIO}_n^-)$, for all  $\alpha\in \mathcal{PIO}_{n-1}^{-}$. 
Moreover, it is routine to show that the mapping $\mathcal{PIO}_{n-1}^-\rightarrow  N(\mathcal{PIO}_{n}^{-})$, 
$\alpha\mapsto\widehat{\alpha}$, is a bijection. 
Thus, by Theorem \ref{12}, $\lvert N(\mathcal{PIO}_n^-) \rvert= \lvert \mathcal{PIO}_{n-1}^- \rvert=2^{n+1}-\frac{(n+1)(n+2)}{2}$.
\end{proof}  

Let $S$ be a finite semigroup $S$ with zero $0$. 
It is well known that $S$ is nilpotent if and only if $S^m=\{0\}$, for some positive integer $m$. 
Therefore, it is easy to prove that a non-trivial finite nilpotent semigroup $S$ is generated by $S\setminus S^2$, 
whence  $S\setminus S^2$ is the minimum generating set of $S$ 
and, consequently, $S$ has rank $\lvert S\rvert - \lvert S^{2}\rvert$. 

\begin{theorem}\label{17}
	For $n\geqslant 2$, the rank of $N(\mathcal{PIO}_{n}^{-})$ is $2^{n}-n-1$.
\end{theorem}

\begin{proof} 
First, notice that, for each $\alpha\in N(\mathcal{PIO}_{n}^{-})^{2}$, 
we must have $x\alpha \leqslant x-2$ for all $x\in \dom(\alpha)$, and so, in particular, 
$1,2\notin \dom(\alpha)$. 
Since $N(\mathcal{PIO}_{2}^{-})=\{0_2,\left(\begin{smallmatrix}2\\1\end{smallmatrix}\right)\}$, it is clear that $N(\mathcal{PIO}_{2}^{-})$ has rank $1=2^2-2-1$. 
Next, suppose that $n\geqslant3$. 
For $\alpha\in \mathcal{PIO}_{n-2}^{-}$, define $\ddot{\alpha}\in\mathcal{PT}_{n}$ by 
$\dom(\ddot{\alpha}) =\{ x\in \Omega_{n} \mid x-2\in \dom(\alpha)\}$, and $x\ddot{\alpha} =(x-2)\alpha$ for $x\in \dom(\ddot{\alpha})$. 
It is not difficult to show that $\ddot{\alpha}\in N(\mathcal{PIO}_n^-)^2$, for all  $\alpha\in \mathcal{PIO}_{n-2}^{-}$, 
and that  the mapping $\mathcal{PIO}_{n-2}^-\rightarrow  N(\mathcal{PIO}_{n}^{-})^2$, 
$\alpha\mapsto\ddot{\alpha}$, is a bijection. 
Thus, by Proposition \ref{16} and Theorem \ref{12}, we get 
$$
|N(\mathcal{PIO}_{n}^{-})|-|N(\mathcal{PIO}_{n}^{-})^2|=\left(2^{n+1}-\frac{(n+1)(n+2)}{2}\right) - \left( 2^{n}-\frac{n(n+1)}{2}\right) =2^{n}-n-1,
$$ 
as required. 
\end{proof}

\medskip 

To end this section, we calculate the rank of $\mathcal{PIO}_{n}^{-}$. 
Since $\mathcal{PIO}_{1}^{-}=\mathcal{PIO}_{1}$, we already noticed that this monoid has rank $1$. 
Now, for $n\geqslant 2$, recall the transformations defined in (\ref{gen}) and consider 
$A=\{a_{1},\ldots ,a_{n-1}\}$, $E=\{e_{1},\ldots,e_{n-1}\}$ and $F=\{f_{2},\ldots,f_{n}\}$. 
It has been shown in \cite[Theorem 3.4]{Fernandes:2024} that $\mathcal{IO}_{n}^{-}= \langle A \rangle$ and each element of $A$ is undecomposable element in $\mathcal{IO}_{n}^{-}$. Our goal is to show that $A\cup E\cup F$ is a generating set of $\mathcal{PIO}_{n}^{-}$ of minimum size. 

\begin{proposition}\label{14}
For $n\geqslant 2$,	$\mathcal{PIO}_{n}^{-}=\langle A\cup E\cup F \rangle$.
\end{proposition}

\begin{proof}  
First, observe that  $0_{n} =e_{1}f_{2}\in \langle A\cup E\cup F\rangle$. Then, take $\alpha\in \mathcal{PIO}_n^-\setminus\{0_n\}$ and define 
a transformation $\beta\in\mathcal{T}_n$  by
\begin{eqnarray*}
	x\beta &=& \left\{ \begin{array}{cl} 
		x\alpha & \mbox{if $x\in\dom(\alpha)$} \\
		\max\{1,\min(\im(\alpha))-\min(\dom(\alpha))+x\} & \mbox{if $x<\min(\dom(\alpha))$} \\
		\max(\im(\alpha))  &\mbox{if $ x>\max(\dom(\alpha))$}
	\end{array}\right.
\end{eqnarray*}
for all $x\in \Omega_{n}$. It is easy to check that $\beta\in\mathcal{IO}_{n}^{-}$, whence $\beta\in\langle A\rangle$. 
On the other hand, it is clear that $\alpha=e_if_j\beta$, with $j=\min(\dom( \alpha))$ and $i=\max(\dom(\alpha))$, 
where $e_n=f_1=1_n$, 
from which follows that $\alpha \in \langle A\cup E\cup F\rangle$, as required. 
\end{proof}  

\begin{theorem}\label{15}
	For $n\geqslant 2$, the rank of the monoid $\mathcal{PIO}_{n}^{-}$ is $3n-3$. 
\end{theorem}

\begin{proof}
In view of Proposition \ref{14}, it suffices to show that each element of $A\cup E \cup F$ is undecomposable in $\mathcal{PIO}_{n}^{-}$. 
Let $1\leqslant i\leqslant n-1$. 

Suppose that $a_i=\alpha\beta$, for some $\alpha,\beta\in\mathcal{PIO}_{n}^{-}$. Then, 
$\dom(a_i)\subseteq\dom(\alpha)$ and $|\im(\alpha)|,|\im(\beta)|\geqslant n-1$. 
In particular, $\alpha\in\mathcal{IO}_{n}^{-}$. 
If $\beta\in\mathcal{IO}_{n}^{-}$, then $\alpha=a_i$ or $\beta=a_i$, since $a_i$ is undecomposable in $\mathcal{IO}_{n}^{-}$. 
So, suppose that $\beta\not\in\mathcal{IO}_{n}^{-}$. It follows that $\alpha\neq 1_n$ and so $\im(\alpha)=[1,n-1]$, which implies that 
 $\dom(\beta)=[1,n-1]$ and so $\im(\beta)=[1,n-1]$. Hence, $\beta=e_{n-1}$ and so $\alpha=a_i$. 
 Thus, we proved that $a_i$ is undecomposable. 
 
Next, suppose that $e_i=\alpha\beta$, for some $\alpha,\beta\in\mathcal{PIO}_{n}^{-}$. Then, 
$[1,i]\subseteq\dom(\alpha)$ and, for all $x\in[1,i]$, 
$x=(x)e_i=x\alpha\beta\leqslant x\alpha\leqslant x$, whence $x\alpha=x=x\beta$. 
In particular, this implies that $[1,i]\subseteq\dom(\beta)$.  
If $i+1\not\in\dom(\alpha)$, then $\dom(\alpha)=[1,i]$ and so $\alpha=e_i$. 
So, suppose that $i+1\in\dom(\alpha)$. If $(i+1)\alpha=i$, then $i+1\in\dom(\alpha\beta)=\dom(e_i)$, 
which is a contradiction. Therefore, $(i+1)\alpha=i+1$ and so $i+1\not\in\dom(\beta)$ (otherwise, we get again $i+1\in\dom(\alpha\beta)=\dom(e_i)$, 
a contradiction). Hence, $\dom(\beta)=[1,i]$ and so $\beta=e_i$. 
Thus, we proved that $e_i$ is undecomposable. Similarly, we can prove $f_{i+1}$ is undecomposable. 
\end{proof}

\section{A presentation for $\mathcal{PIO}_n^-$} 

Throughout this section we consider $n\geqslant2$. 

\smallskip 

Recall that,  by \cite[Theorem 3.4]{Fernandes:2024}, $A$ is a generating set of $\mathcal{IO}_{n}^{-}$ of minimum size. Moreover, by considering the set $A$ as an alphabet with $n-1$ letters and the set of $\frac{1}{2}(n^2-n)$ monoid relations 
\begin{description}
	
	\item $(R_1)$ $a_ia_{n-1}=a_i$, for $1\leqslant i\leqslant n-1$, and  

         \item\hspace*{22pt} $a_ia_j=a_{j+1}a_i$, for $1\leqslant i\leqslant j\leqslant n-2$,  
\end{description} 
by \cite[Theorem 3.6]{Fernandes:2024}, the presentation $\langle A\mid R_1\rangle$ defines the monoid $\mathcal{IO}_{n}^{-}$.

\smallskip 

Now, let us consider the set $A\cup E\cup F$ as an alphabet with $3n-3$ letters and let $R$ be the set 
of size $\frac{1}{2}(5n^2+9n-20)$ formed by $R_1$ and the following monoid relations:

\begin{description}
	
	\item  $(R_2)$ $e_{i}^2=e_{i}$ and $f_{i+1}^2=f_{i+1}$, for $1\leqslant i\leqslant n-1$; 
	
	\item  $(R_3)$ $e_if_j=f_je_i$, for $2\leqslant j\leqslant i\leqslant n-1$; 
	
	\item\hspace*{22pt} $e_{i}f_{i+1}=f_{i+1}e_{i}=e_1f_n$, for $1\leqslant i \leqslant n-1$; 
	
	\item\hspace*{22pt} $e_ie_{i+1}=e_{i+1}e_i=e_i$, for $1\leqslant i\leqslant n-2$; 
	
	\item\hspace*{22pt} $f_if_{i+1}=f_{i+1}f_i=f_{i+1}$, for $2\leqslant  i\leqslant n-1$; 
	
	\item  $(R_4)$  $a_{i}e_{i-1}=e_{i-1}$, for $2\leqslant i\leqslant n-1$; 
	
	\item\hspace*{22pt} $a_ie_j=e_{j+1}a_i$, for $1\leqslant i\leqslant j\leqslant n-2$; 
	
	\item\hspace*{22pt} $a_{n-1}e_{n-1}=a_{n-1}$ and $e_ia_i=e_i$, for $1\leqslant i\leqslant n-1$; 

	\item  $(R_5)$  $a_if_j=f_ja_i$, for $2\leqslant j\leqslant i\leqslant n-1$; 
	
	\item\hspace*{22pt} $a_if_j=f_{j+1}a_i$, for $1\leqslant i < j\leqslant n-1$; 
	
	\item\hspace*{22pt} $a_{1}f_n=e_1f_n$ and $a_{n-1}f_n=e_1f_n$; 
	
	\item  $(R_6)$  $f_{i+1}a_i=f_{i+1}a_1$, for $2\leqslant i\leqslant n-1$.  
	
\end{description}

\smallskip

We can immediately derive from $R$ the following relations. 

\begin{lemma}\label{morerel}
$\nonumber$ One has: 
\begin{description}
	\item  $(R'_3)$ $e_if_j \sim_R f_je_i$, for $1\leqslant i\leqslant n-1$ and $2\leqslant j\leqslant n$; 

         \item\hspace*{22pt} 	$e_{i}f_{j} \sim_R e_{1}f_{n}$, for $1\leqslant i<j \leqslant n$; 
         
         \item\hspace*{22pt} $e_ie_j \sim_R e_je_i \sim_R e_i$, for $1\leqslant i < j\leqslant n-1$; 
         
         \item\hspace*{22pt} $f_if_j \sim_R f_jf_i \sim_R f_j$, for $2\leqslant i < j\leqslant n$; 
         
         \item  $(R'_4)$ $a_ie_j \sim_R e_j$, for $1\leqslant j <i\leqslant n-1$; 
                  
         \item\hspace*{22pt} $e_ja_i\sim_R e_j$, for $1\leqslant j \leqslant i\leqslant n-1$; 
          
         \item\hspace*{22pt} $a_ie_{n-1} \sim_R a_i$, for $1\leqslant i\leqslant n-1$;
         
         \item  $(R'_5)$ $a_if_n \sim_R e_1f_n$, for $1\leqslant i\leqslant n-1$; 
         
         \item  $(R'_6)$ $f_ja_i \sim_R f_ja_1$, for $2\leqslant i <j \leqslant n$. 
\end{description} 
\end{lemma}
\begin{proof}
($R'_3$) First, let us show by induction on $k$ that $e_ie_{i+k} \sim_R e_{i+k}e_i \sim_R e_i$, 
for $1\leqslant i\leqslant n-2$ and $1\leqslant k\leqslant n-1-i$. 
For $k=1$, the relations follow directly from $R_3$. Now, admit that the relations are valid for some $k\geqslant1$ (and $k< n-1-i$).  
Then, by applying the induction hypothesis in the first and third steps, we get 
$$
e_ie_{i+k+1} \sim_R  e_i e_{i+k} e_{i+k+1}  \sim_{R_3} e_i e_{i+k}  \sim_R  e_i 
$$
and, similarly, $e_{i+k+1}e_i \sim_R e_i$. 
Thus, we have $e_ie_j \sim_R e_je_i \sim_R e_i$, for $1\leqslant i < j\leqslant n-1$. 

By an analogous process of induction, we can also show that 
$f_if_j \sim_R f_jf_i \sim_R f_j$, for $2\leqslant i < j\leqslant n$, 
and $e_{i}f_{j} \sim_R f_je_i \sim_R e_{1}f_{n}$, for $1\leqslant i<j \leqslant n$. 

Finally, observe that we have also obtained for free that $e_if_j \sim_R f_je_i$, for $1\leqslant i\leqslant n-1$ and $2\leqslant j\leqslant n$. 

\smallskip 

($R'_4$) Let  $1\leqslant j <i\leqslant n-1$. Then, by $R'_3$ and $R_2$, we have $e_{i-1}e_j\sim_R e_j$, whence
$$
a_ie_j\sim_R a_i e_{i-1}e_j \sim_{R_4} e_{i-1}e_j \sim_R e_j. 
$$

Next, let $1\leqslant j \leqslant i\leqslant n-1$. Then, by $R'_3$ and $R_2$, we have $e_je_i\sim_R e_j$ and so 
$$
e_ja_i\sim_R e_je_ia_i \sim_{R_4} e_je_i \sim_R e_j. 
$$

At last, let $1\leqslant i\leqslant n-1$. Then, 
$$
a_ie_{n-1}\sim_{R_1}a_ia_{n-1}e_{n-1}\sim_{R_4}a_ia_{n-1}\sim_{R_1} a_i. 
$$

\smallskip 

($R'_5$)   Let  $1\leqslant i\leqslant n-1$. 
Since $a_{1}f_n\sim_{R_5} e_1f_n$, we can suppose $2\leqslant i\leqslant n-1$. 
Then, by $R'_4$, we have $a_ie_1\sim_R e_1$ and so 
$$
a_if_n\sim_{R_1}a_ia_{n-1}f_n \sim_{R_5}a_ie_1f_n \sim_R e_1f_n. 
$$

\smallskip 

($R'_6$)   Let  $2\leqslant i < j \leqslant n$. Then, by $R'_3$ and $R_2$, we have $f_jf_{i+1}\sim_R f_j$, whence 
$$
f_ja_i\sim_R f_jf_{i+1}a_i \sim_{R_6} f_jf_{i+1}a_1 \sim_R f_ja_1,
$$
as required.
\end{proof} 
 
\smallskip 

Let $\varphi:(A\cup E\cup F)^*\longrightarrow \mathcal{PIO}_{n}^-$ be the homomorphism of monoids 
that extends the mapping $A\cup E\cup F\longrightarrow \mathcal{PIO}_{n}^-$ defined by $a_{i}\mapsto a_{i}$, $e_{i}\mapsto e_{i}$ and $f_{i+1}\mapsto f_{i+1}$,  
for $1\leqslant i\leqslant n-1$. 

\smallskip 

It is a routine matter to prove the following lemma: 

\begin{lemma} \label{satisfies} 
The generating set $A\cup E\cup F$ of $\mathcal{PIO}_{n}^-$ satisfies (via $\varphi$) all relations from $R$. 
\end{lemma}

\smallskip 

Next, recall that 
$$
W^-=\{a_{i_k}\cdots a_{i_1}\in A^*\mid 0\leqslant k\leqslant n-1,~ 1\leqslant i_1<\cdots<i_k\leqslant n-1\}
$$
is a set of canonical forms of $\mathcal{IO}_{n}^-$; see \cite{Fernandes:2024}. 

Let $f_1$ and $e_n$ be both the empty word of $(A\cup E\cup F)^*$. 
Then, we have the following:

\begin{lemma} \label{funda} 
Let $w\in (A\cup E\cup F)^*$. Then, $w\sim_R e_i f_j a_{i_k}\cdots a_{i_1}$, 
for some $0\leqslant k\leqslant n-1$, $1\leqslant i_1<\cdots<i_k\leqslant n-1$ and $1\leqslant j\leqslant i\leqslant n$ or $(i,j)=(1,n)$. 
\end{lemma} 

\begin{proof}
It is clear from $R_4$, $R_5$, $R'_4$ and $R'_5$ and that there exist $u\in (E\cup F)^*$ and $v\in A^*$ such that $w\sim_R uv$. 
Since $W^-$ is a set of canonical forms of $\mathcal{IO}_{n}^-$ and $\langle A\mid R_1\rangle$ is a presentation for $\mathcal{IO}_{n}^-$, there exist $0\leqslant k\leqslant n-1$ and $1\leqslant i_1<\cdots<i_k\leqslant n-1$ such that $v\sim_R a_{i_k}\cdots a_{i_1}$. 
Moreover, it is also clear from $R_2$, $R_3$ and $R'_3$ that $u\sim_R e_if_j$ for some $1\leqslant i,j\leqslant n$. 
If $i<j$, then it follows from $R'_3$ that $e_if_j\sim_R e_1f_n$. Otherwise, we have $1\leqslant j\leqslant i\leqslant n$, as required.  
\end{proof} 

Let $w\in (A\cup E\cup F)^*$ and $\alpha=w\varphi$. Suppose that $w\sim_R e_i f_j a_{i_k}\cdots a_{i_1}$, with  
$0\leqslant k\leqslant n-1$, $1\leqslant i_1<\cdots<i_k\leqslant n-1$ and $1\leqslant j\leqslant i\leqslant n$ or $(i,j)=(1,n)$. Then, 
$\alpha=1_n$ if and only if $w$ is the empty word; 
$\alpha=0_n$ if and only if $(i,j)=(1,n)$; 
and if $1\leqslant j\leqslant i\leqslant n$, then $\dom\alpha=[j,i]$. Observe also that, for $k\geqslant1$, 
$$
\mbox{\small	
$\begin{array}{c}
\hspace{-6mm}	(a_{i_k}\cdots a_{i_1})\varphi = 
	\left(\begin{array}{ccc|ccc|c|ccc|c} 
		1 &\cdots & i_1 & i_1+1 &\cdots &  i_2  &\cdots &  i_t+1  &\cdots &  i_{t+1}  &\cdots \\
		1 &\cdots & i_1 &  i_1  &\cdots & i_2-1 &\cdots & i_t-t+1 &\cdots & i_{t+1}-t &\cdots
	\end{array}\right.\vspace{2mm} \\ 
\hspace{42mm}	 \left. \begin{array}{c|ccc|ccc}
		\cdots &  i_{k-1}+1  &\cdots &   i_k   &  i_k+ 1 &\cdots & n\\
		\cdots & i_{k-1}-k+2 &\cdots & i_k-k+1 & i_k-k+1 &\cdots & n-k
	\end{array}\right),
\end{array}$ 
} 
$$
i.e., with $i_0=0$ and $i_{k+1}=n$, we have 
\begin{equation}\label{e12}
\mbox{$(x)(a_{i_k}\cdots a_{i_1})\varphi =x-t$, for $x\in [i_t+1,i_{t+1}]$ and $0\leqslant t\leqslant k$}.  
\end{equation}

\begin{lemma} \label{aux} 
If $1\leqslant i_1<\cdots<i_t <j \leqslant n$, then $f_j a_{i_t}\cdots a_{i_1}\sim_R f_ja_1^t$. 
\end{lemma} 

\begin{proof}
We will proceed by induction on $t$. For $t=1$, we have $f_ja_{i_1}\sim f_ja_1$, by $R'_6$. Let us assume that the result is valid for $t-1$, with $t\geqslant 2$. 
Let $1\leqslant i_1<\cdots<i_t <j \leqslant n$. Then, $i_{t-1}<j-1$ and so, by induction hypothesis, $f_{j-1} a_{i_{t-1}}\cdots a_{i_1}\sim_R f_{j-1}a_1^{t-1}$. 
Since $j>t\geqslant 2$, we have 
\begin{eqnarray*}
	f_j a_{i_t}a_{i_{t-1}}\cdots a_{i_1} &\sim_{R_6}& f_j a_1a_{i_{t-1}}\cdots a_{i_1}\sim_{R_5} a_1f_{j-1}a_{i_{t-1}}\cdots a_{i_1} \\ 
	&\sim_R& a_1f_{j-1}a_1^{t-1} \sim_{R_5}  f_j a_1a_1^{t-1}=f_j a_1^t,
\end{eqnarray*}
as required. 
\end{proof}

\begin{theorem}\label{presPIO-} 
The monoid $\mathcal{PIO}_{n}^-$ is defined by the presentation $\langle A\cup E\cup F\mid R\rangle$ on $3n-3$ generators and $\frac{1}{2}(5n^2+9n-20)$ relations. 
\end{theorem} 

\begin{proof}
By Lemma \ref{satisfies}, it remains to prove that the second condition of Proposition \ref{provingpresentation} is satisfied. 
Let us take $w,w'\in (A\cup E\cup F)^*$ such that $w\varphi=w'\varphi$. We aim to show that $w\sim_R w'$.  

If $w\varphi=1_n$, then $w$ and $w'$ are both the empty word, whence $w\sim_R w'$. 

Suppose that $w\varphi\neq 1_n$. By Lemma \ref{funda}, there exist $1\leqslant k\leqslant n-1$, $1\leqslant r_1<\cdots<r_k\leqslant n-1$ and $1\leqslant i,j \leqslant n$ such that $w\sim_R e_i f_j a_{r_k}\cdots a_{r_1}$ with $j\leqslant i$ or $(i,j)=(1,n)$. Moreover, since $\dom(w'\varphi)= \dom(w\varphi)=[j,i]$, there exist  $1\leqslant \ell\leqslant n-1$ and $1\leqslant s_1<\cdots<s_\ell\leqslant n-1$ such that $w'\sim_R e_{i} f_{j} a_{s_\ell}\cdots a_{s_1}$. If $w\varphi=0_n$, then $(i,j)=(1,n)$ and so 
$$
w\sim_R e_1f_na_{r_k}\cdots a_{r_1}\sim_{R_3} f_ne_1a_{r_k}\cdots a_{r_1} \sim_R f_ne_1\sim_R e_1f_n. 
$$
(by applying $R'_4$ and $R'_3$, respectively, in the last two steps). 
Similarly, we get $w'\sim_R e_1f_n$ and so $w\sim_R w'$.

Suppose $w\varphi\neq 0_n$. Then, $\dom(w\varphi)=[j,i]\not= \emptyset$, i.e. $j\leqslant i$, 
and so, in view of $(\ref{e12})$, there exist $0\leqslant p\leqslant q\leqslant k$ and $0\leqslant p'\leqslant q'\leqslant \ell$ such that 
$$
j\in [r_p+1,r_{p+1}]\cap [s_{p'}+1,s_{p'+1}] \mbox{ and }\, i\in [r_q+1,r_{q+1}]\cap [s_{q'}+1,s_{q'+1}] 
$$
where $r_0=s_0=0$ and $r_{k+1}=s_{\ell+1}=n$. 
Then, we have $(x)(a_{r_k} \cdots a_{r_1}) =(x)(w\varphi) =(x)(w'\varphi) =(x)(a_{s_\ell} \cdots a_{s_1})$ for all $x\in [j,i]$. 
In particular, $(j)(a_{r_k} \cdots a_{r_1})= j-p=j-p' =(j)(a_{s_\ell} \cdots a_{s_1})$ and so $p=p'$. 
Similarly, we have $q=q'$. 
Next, we show that $r_t=s_t$ for all $p+1\leqslant t\leqslant q$. 
Let $p+1\leqslant t\leqslant q$, Since $j\leqslant r_{p+1}\leqslant r_{t}\leqslant r_{q} \leqslant i-1$, we get $r_{t}\in [j,i]$. 
Similarly, we also get $s_{t}\in [j,i]$. 
Moreover, since $s_{t}\in [s_{t-1}+1,s_{t}]$ and $s_{t}\in [r_{u-1}+1,r_{u}]$ for some $p+1\leqslant u\leqslant q+1$, 
we have $s_t-t+1=(s_t)(a_{s_\ell} \cdots a_{s_1})=(s_t)(a_{r_k} \cdots a_{r_1})=s_t-u$, whence $u=t-1$. 
Therefore, $s_{t}\in [r_{t-1}+1,r_{t}]$ and so $s_{t}\leqslant r_{t}$. 
Similarly, we obtain $r_{t}\leqslant s_{t}$. Thus, $a_{r_q}\cdots a_{r_{p+1}} =a_{s_q}\cdots a_{s_{p+1}}$.

Now, since $r_1<\cdots<r_p<j\leqslant r_{p+1}<\cdots<r_k$, we have 
$$
f_ja_{r_k}\cdots a_{r_{p+1}} a_{r_p}\cdots a_{r_1} \sim_{R_5} a_{r_k}\cdots a_{r_{p+1}} f_ja_{r_p} \cdots a_{r_1}
\sim_R  a_{r_k}\cdots a_{r_{p+1}} f_j a_1^p,  
$$
by applying Lemma \ref{aux} in the last step. 
Similarly, $f_j a_{s_\ell}\cdots a_{s_1} \sim_R a_{s_\ell}\cdots a_{s_{p+1}} f_j a_1^p$. 
Moreover, since $i\leqslant r_{q+1}<\cdots<r_k$ and $i\leqslant s_{q+1}<\cdots<s_\ell$, it follows from $R'_4$ that 
$$
e_i a_{r_k}\cdots a_{r_{p+1}} \sim_R e_i a_{r_q}\cdots a_{r_{p+1}} 
\quad \text{and}\quad 
e_i a_{s_\ell}\cdots a_{s_{p+1}} \sim_R e_i a_{s_q}\cdots a_{s_{p+1}}. 
$$
Therefore, 
\begin{align*}
w\sim_R e_if_ja_{r_k}\cdots a_{r_1} \sim_R e_i a_{r_k}\cdots a_{r_{p+1}} f_j a_1^p \sim_R 
e_i a_{r_q}\cdots a_{r_{p+1}} f_j a_1^p  \qquad\quad \\ = e_i a_{s_q}\cdots a_{s_{p+1}} f_j a_1^p 
\sim_R e_i a_{s_\ell}\cdots a_{s_{p+1}} f_j a_1^p \sim_R e_if_j a_{s_\ell}\cdots a_{s_1} \sim_R w', 
\end{align*}
as required. 
\end{proof} 

\bigskip 

To finish this section, by using the isomorphism $\phi$, we will deduce an immediate presentation for $\mathcal{PIO}_{n}^+$. The relations obtained in this way will be used, in the next section, to obtain a presentation for $\mathcal{PIO}_{n}$. 

\smallskip 

Observe that the transformations $b_{1},\ldots ,b_{n-1}$ defined in (\ref{gen}) belong to $\mathcal{PIO}_n^+$. 
Moreover, we have 
$$
{a}_i\phi =b_{n-i}, \quad e_i\phi= f_{n+1-i} \quad \mbox{and} \quad  f_{i+1}\phi =e_{n-i},
$$
for $1\leqslant i\leqslant n-1$. Thus, if $B=\{b_{1},\ldots ,b_{n-1}\}$, then $B\cup E\cup F$ is a generating set of  $\mathcal{PIO}_{n}^+$ with minimum size. 

\smallskip 

Next, consider the following sets of monoid relations over the alphabet $B\cup E\cup F$:

\begin{description}
	
	\item $(R_\text{1b})$ $b_ib_1=b_i$, for $1\leqslant i\leqslant n-1$; 

        \item\hspace*{22pt} $b_ib_{j+1}=b_{j+1}b_{i+1}$, for $1\leqslant i\leqslant j\leqslant n-2$; 
    
        \item  $(R_\text{3b})$ $e_if_j=f_je_i$, for $2\leqslant j\leqslant i\leqslant n-1$; 
	
	\item\hspace*{22pt} $e_{i}f_{i+1}=f_{i+1}e_{i}=f_ne_1$, for $1\leqslant i \leqslant n-1$; 
	
	\item\hspace*{22pt} $e_ie_{i+1}=e_{i+1}e_i=e_i$, for $1\leqslant i\leqslant n-2$; 
	
	\item\hspace*{22pt} $f_if_{i+1}=f_{i+1}f_i=f_{i+1}$, for $2\leqslant  i\leqslant n-1$; 
    
	\item  $(R_\text{4b})$  $b_if_{i+2}=f_{i+2}$, for $1\leqslant i\leqslant n-2$; 
	
	\item\hspace*{22pt} $b_if_{j+1}=f_{j}b_i$, for $2\leqslant j\leqslant i\leqslant n-1$; 
	
	\item\hspace*{22pt} $b_1f_{2}=b_1$ and $f_{i+1}b_i=f_{i+1}$, for $1\leqslant i\leqslant n-1$; 
		
	\item  $(R_\text{5b})$  $b_ie_j=e_jb_i$, for $1\leqslant i<j \leqslant n-1$; 
	
	\item\hspace*{22pt} $b_ie_j=e_{j-1}b_i$, for $2\leqslant  j\leqslant i \leqslant n-1$; 
	
	\item\hspace*{22pt} $b_1e_1=f_ne_1$ and $b_{n-1}e_1=f_ne_1$; 
   
   	\item  $(R_\text{6b})$  $e_ib_i=e_ib_{n-1}$, for $1\leqslant i \leqslant n-2$.  

\end{description}

\smallskip

Let $R_\text{b}=R_\text{1b}\cup R_2\cup R_\text{3b}\cup R_\text{4b}\cup R_\text{5b}\cup R_\text{6b}$. 
Then, as an immediate consequence of Theorem \ref{presPIO-}, we have the following corollary. 

\begin{corollary}\label{presPIO+} 
The monoid $\mathcal{PIO}_{n}^+$ is defined by the presentation $\langle B\cup E\cup F\mid R_\text{b}\rangle$ on $3n-3$ generators and $\frac{1}{2}(5n^2+9n-20)$ relations. 
\end{corollary} 

\section{A presentation for $\mathcal{PIO}_{n}$} 

As in the previous section, throughout this section we also consider $n\geqslant2$. 

\smallskip 

We begin this section with recalling the presentation for $\mathcal{IO}_n$ exhibited by Fernandes in \cite{Fernandes:2024}. Consider the following set of monoid relations over the alphabet $A\cup B$: 

\begin{description}
		
	\item  $(R_7)$  $b_ia_j=a_jb_i$, for $1\leqslant i< j\leqslant n-1$; 
	
	\item\hspace*{22pt} $b_ia_1=a_i$, for $1\leqslant i\leqslant n-1$;  
	
	\item\hspace*{22pt} $a_ib_{n-1}=b_i$, for $1\leqslant i\leqslant n-1$. 
		
\end{description}

Notice that $a_{1},\ldots ,a_{n-1}, b_{1},\ldots ,b_{n-1}\in \mathcal{IO}_n$. Moreover, by \cite[Proposition 3.3 and Theorem 3.5]{Fernandes&Paulista:2023}, for $n\geqslant 3$, $\{ a_{1},\ldots ,a_{n-2},b_{n-1} \}$ is a generating set of $\mathcal{IO}_n$ of minimum size. In particular, for $n\geqslant 3$, the monoid $\mathcal{IO}_n$ has rank $n-1$. 
On the other hand, by \cite[Theorem 4.8]{Fernandes:2024}, the monoid $\mathcal{IO}_n$ is defined by the presentation $\langle A\cup B \mid R_1\cup R_\text{1b} \cup R_7 \rangle$ on $2n-2$ generators and $\frac{1}{2}(3n^2-n-2)$ relations. 

\smallskip 

Observe that, since $\{a_1,\ldots,a_{n-2},b_{n-1},e_{n-1},f_2\}$ is a generating set of the monoid $\mathcal{PIO}_n$, $A\cup B\cup E\cup F$ is obviously a generating set of $\mathcal{PIO}_n$ as well. 

Let us consider the set 
$$
\bar{R}=R_1\cup R_\text{1b}\cup R_2\cup R_3\cup R_4\cup R_\text{4b}\cup R_5\cup R_\text{5b}\cup R_\text{6}\cup R_\text{6b}\cup R_7
$$ 
of $5n^2+3n-10$ monoid relations over the alphabet $A\cup B\cup E\cup F$. 

Let us observe that, by Lemma \ref{morerel}, we have $e_1f_n\sim_R f_ne_1$, whence $e_1f_n\sim_{\bar{R}} f_ne_1$ and so 
the congruences 
$\sim_{\bar R}$ and $\sim_{\bar R\cup R_\text{3b}}$ of $(A\cup B\cup E\cup F)^*$ coincide. 

Let $\psi:(A\cup B\cup E\cup F)^*\longrightarrow \mathcal{PIO}_{n}$ be the homomorphism of monoids that extends the mapping $A\cup B\cup E\cup F\longrightarrow \mathcal{PIO}_{n}$ defined by $a_{i}\mapsto a_{i}$, $b_{i}\mapsto b_{i}$, $e_{i}\mapsto e_{i}$ and $f_{i+1}\mapsto f_{i+1}$ for $1\leqslant i\leqslant n-1$. 

\smallskip 

It is a routine matter to prove the following lemma: 

\begin{lemma} \label{satisfies2} 
The generating set $A\cup B\cup E\cup F$ of $\mathcal{PIO}_{n}$ satisfies (via $\psi$) all relations from $\bar R$. 
\end{lemma} 

Now, recall that 
\begin{eqnarray*}
W=\{a_{i_k}\cdots a_{i_1}b_{n-1}^\ell &\mid & 1\leqslant i_1< \cdots <i_k \leqslant \min\{n-\ell+k-2,n-1\},  \\
&& 0\leqslant k\leqslant n-1 \mbox{ and } 0\leqslant \ell\leqslant n-1\}.
\end{eqnarray*}
is a set of canonical forms of $\mathcal{IO}_{n}$; see \cite{Fernandes:2024}. 

With a proof similar to Lemma \ref{funda}, 
using additionally the relations $R_\text{4b}$, $R_\text{5b}$ and the relations obtained from $R'_4$ and $R'_5$ (Lemma \ref{morerel}) 
using the isomorphism $\phi$, as well as the set of canonical forms $W$ and the fact that 
$\langle A\cup B\mid R_1\cup R_\text{1b}\cup R_7\rangle$ is a presentation for $\mathcal{IO}_{n}$, we have:

\begin{lemma} \label{funda2} 
Let $w\in (A\cup B\cup E\cup F)^*$. Then, $w\sim_{\bar R} e_i f_j a_{i_k} \cdots a_{i_1} b_{n-1}^\ell$, 
for some $1\leqslant i_1< \cdots< i_k\leqslant \min\{n-\ell+k-2,n-1\}$, 
$0\leqslant k\leqslant n-1$, $0\leqslant \ell\leqslant n-1$ and $1\leqslant j\leqslant i\leqslant n$ or $(i,j)=(1,n)$. 
\end{lemma} 

Next, we prove a series of auxiliary results. 

\begin{lemma}\label{r1}
If $1\leqslant j\leqslant n$, then $f_ja_1^{j-1}b_{n-1}^{j-1} \sim_{\bar R} f_j$. 
\end{lemma}
\begin{proof}
We proceed by induction on $j$. For $j=1$, the relation is trivial. Let $1\leqslant j <n$ and admit that 
$f_j\sim_{\bar R} f_ja_1^{j-1}b_{n-1}^{j-1}$.  
Then, 
\begin{align*}
f_{j+1}a_1^jb_{n-1}^j=f_{j+1}a_1^{j-1}(a_1b_{n-1})b_{n-1}^{j-1} 
\sim_{R_7} f_{j+1}a_1^{j-1} b_1 b_{n-1}^{j-1} \sim_{R_3\cup R_\text{1b}}  \qquad \\ 
f_{j+1}(f_ja_1^{j-1} b_{n-1}^{j-1})b_j \sim_{\bar R} f_{j+1}f_jb_j\sim_{R_3}f_{j+1}b_j \sim_{R_\text{4b}} f_{j+1}, 
\end{align*} 
as required. 
\end{proof} 

\begin{lemma}\label{r2}
If $2\leqslant j\leqslant n$, then 
$f_ja_1^{j-1}b_{n-1}^{j-2}\sim_{\bar R} f_ja_1$. 
\end{lemma}
\begin{proof}
For $j=2$, the relation is identical. For $3\leqslant j \leqslant n$, 
by Lemma \ref{r1}, we have $f_{j-1} a_1^{j-2} b_{n-1}^{j-2}\sim_{\bar R} f_{j-1}$, whence 
$$
f_ja_1^{j-1} b_{n-1}^{j-2}= (f_ja_1) a_1^{j-2} b_{n-1}^{j-2} \sim_{R_5} a_1(f_{j-1} a_1^{j-2} b_{n-1}^{j-2}) \sim_{\bar R} a_1f_{j-1} \sim_{R_5} f_ja_1, 
$$
as required.  
\end{proof} 

Let $1\leqslant p<j\leqslant n$. Then, it is a routine matter to check that the equality 
$a_1^{j-1}b_{n-1}^{j-2}a_1^{p-1} =  a_1^{j-1}b_{n-1}^{j-p-1}$ is valid in $\mathcal{IO}_n$. 
Since $\mathcal{IO}_n$ is defined by the presentation $\langle A\cup B\mid R_1\cup R_\text{1b} \cup R_7\rangle$ 
and  $R_1\cup R_\text{1b}\cup R_7\subseteq \bar{R}$, we immediately have the following lemma. 

\begin{lemma}\label{r3}
If $1\leqslant p<j\leqslant n$, then $a_1^{j-1}b_{n-1}^{j-2}a_1^{p-1}\sim_{\bar R}  a_1^{j-1}b_{n-1}^{j-p-1}$. 
\end{lemma}

\begin{lemma}\label{r4}
If $0\leqslant p\leqslant j-1\leqslant n-1$, then $f_ja_1^p\sim_{\bar R} f_ja_1^{j-1}b_{n-1}^{j-p-1}$.
\end{lemma}
\begin{proof}
If $p=0$, by Lemma \ref{r1}, the relation is true. So, let us suppose that $p\geqslant1$. Then, $j\geqslant2$ and, 
from Lemmas \ref{r2} and \ref{r3}, it follows 
$$
f_ja_1^p=f_ja_1a_1^{p-1} \sim_{\bar R} 
f_ja_1^{j-1}b_{n-1}^{j-2}a_1^{p-1} \sim_{\bar R} 
f_ja_1^{j-1}b_{n-1}^{j-p-1},
$$
as required. 
\end{proof} 

\begin{lemma}\label{r6}
Let $1\leqslant i_1<\cdots<i_k < i\leqslant n$, $0\leqslant k\leqslant n-1$ 
and $0\leqslant \ell\leqslant n-1$ be such that $n-i+k < \ell \leqslant n-i_k+k-1$. 
Then,  $a_{i-1}\cdots a_{n-\ell+k} a_{i_k}\cdots a_{i_1}b_{n-1}^\ell \sim_{\bar R} a_{i_k}\cdots a_{i_1}b_{n-1}^\ell$. 
\end{lemma}

\begin{proof}
By observing that 
$$
(n-\ell+k) a_{i_k}\cdots a_{i_1}b_{n-1}^\ell=n
~\text{ and }~  
(a_{i-1}\cdots a_{n-\ell+k})_{|_{\{1,\ldots,n-\ell+k\}}}=\id_{\{1,\ldots,n-\ell+k\}}, 
$$ 
it is a routine matter to show the equality $a_{i-1}\cdots a_{n-\ell+k} a_{i_k}\cdots a_{i_1}b_{n-1}^\ell = a_{i_k}\cdots a_{i_1}b_{n-1}^\ell$ in $\mathcal{IO}_{n}$. Then the result follows using the same argument as for Lemma \ref{r3}.
\end{proof} 

\medskip 

Let 
\begin{align*}
\overline{W}=\{e_if_ja_{i_k}\cdots a_{i_1}a_1^{j-1}b_{n-1}^\ell\mid 
1\leqslant j\leqslant i\leqslant n,~ j\leqslant i_1<\cdots<i_k\leqslant i-1, \qquad\qquad \\
\mbox{$0\leqslant k\leqslant i-j$ and $0\leqslant \ell\leqslant n-i+j+k-1$}\} 
\cup \{e_1f_n\}.
\end{align*}

We will show that $\overline{W}$ is a set of canonical forms for  $\mathcal{PIO}_{n}$. First,  
observe that 
$$
|\overline{W}|=1 + \sum_{j=1}^{n} \sum_{i=j}^{n} \sum_{k=0}^{i-j}\binom{i-j}{k}(n-i+j+k).
$$ 
Then, by (\ref{e1}), it follows
\begin{eqnarray*}
\lvert \overline{W}\rvert &=&1+\sum_{j=1}^{n} \sum_{i=j}^{n} (2n-i+j)2^{i-j-1} =1+ \sum_{j=1}^{n} \sum_{i=0}^{n-j} (2n-i)2^{i-1}\\
		&=&1+\sum_{j=1}^{n}( n\sum_{i=0}^{n-j} 2^{i} -\frac{1}{2}\sum_{i=0}^{n-j} i2^{i})
		=1+\sum_{j=1}^{n}\left( (n+j+1)2^{n-j} - (n+1)\right)\\
		&=&1-n^2-n+\sum_{j=1}^{n}(n+j+1)2^{n-j} = 1-n^2-n+\sum_{j=0}^{n-1}(2n+1-j)2^{j}\\
		&=&	1-n^2-n+(2n+1)\sum_{j=0}^{n-1}2^{j}-\sum_{j=0}^{n-1}j2^{j}=(n+3)2^n -n^2-3n-2\\
		&=& |\mathcal{PIO}_{n}|. 
\end{eqnarray*}

We are now in a position to establish our presentation for $\mathcal{PIO}_{n}$. 

\begin{theorem}\label{presPIO} 
The monoid $\mathcal{PIO}_{n}$ is defined by the presentation $\langle A\cup B\cup E\cup F\mid \bar R\rangle$ 
on $4n-4$ generators and $5n^2+3n-10$  relations. 
\end{theorem} 
\begin{proof}
In view of Lemma \ref{satisfies2} and the previous observation, it remains to show the second condition of Proposition \ref{ruskuc} (regarding $\overline{W}$). 

Let $w\in (A\cup B\cup E\cup F)^*$. Then, by Lemma \ref{funda2}, there exist 
$0\leqslant k\leqslant n-1$, $0\leqslant \ell\leqslant n-1$, 
$1\leqslant i_1<\cdots<i_k\leqslant \min\{n-\ell+k-2,n-1\}$  and 
$1\leqslant  i,j\leqslant n$ such that $w\sim_{\bar R} e_i f_j a_{i_k}\cdots a_{i_1}b_{n-1}^\ell$ with $j\leqslant i$ or $(i,j)=(1,n)$. 

We get $e_1a_1\sim_R e_1$ as a particular instance of $R'_4$ and so, by applying the isomorphism $\phi$, 
we obtain $f_nb_{n-1}\sim_\text{b}f_n$,  
from which follows $f_nb_{n-1}^\ell \sim_{\bar R} f_n$. Therefore, 
if $(i,j)=(1,n)$, then 
$$
\begin{array}{rcll}
w &\sim_{\bar R}& e_1f_n a_{i_k} \cdots a_{i_1} b_{n-1}^\ell & \\ 
& \sim_{\bar R} & f_ne_1 a_{i_k} \cdots a_{i_1} b_{n-1}^\ell  & \mbox{(by $R'_3$)}\\ 
&\sim_{\bar R}& f_n e_1 b_{n-1}^\ell & \mbox{(by $R'_4$)}\\ 
&\sim_{\bar R}& e_1 f_n b_{n-1}^\ell & \mbox{(by $R'_3$)}\\ 
&\sim_{\bar R} & e_1f_n \in\overline{W}. & 
\end{array}
$$

Suppose that $1\leqslant j\leqslant i\leqslant n$. With $i_0=0$ and $i_{k+1}=n$, let $0\leqslant p, q\leqslant k$ be such that $j\in [i_p+1,i_{p+1}]$ and $i\in [i_q+1,i_{q+1}]$, i.e. 
$$
1\leqslant i_1<\cdots<i_p<j\leqslant i_{p+1}<\cdots<i_q<i\leqslant i_{q+1}<\cdots<i_k\leqslant n-1. 
$$
Then, $0\leqslant p\leqslant j-1$, $0\leqslant q\leqslant i-1$, $p\leqslant q$ and, as in the proof of Theorem \ref{presPIO-}, 
we have  
\begin{equation}\label{asinpio-} 
e_i f_j a_{i_k}\cdots a_{i_1} \sim_{\bar R} e_i a_{i_q}\cdots a_{i_{p+1}} f_j a_1^p. 
\end{equation}
Hence, 
\begin{eqnarray}\label{byasinpio-}
w &\sim_{\bar R}& e_i a_{i_q} \cdots a_{i_{p+1}} f_j a_1^p b_{n-1}^{\ell} \nonumber \\
  &\sim_{\bar R}& e_i a_{i_q} \cdots a_{i_{p+1}} f_j a_1^{j-1} b_{n-1}^{\ell+j-p-1}  \qquad (\mbox{by Lemma \ref{r4})} \nonumber \\ 
  &\sim_{R_5}& e_if_j a_{i_q} \cdots a_{i_{p+1}} a_1^{j-1} b_{n-1}^{\ell+j-p-1} =w',  
\end{eqnarray}
with $j\leqslant i_{p+1}<\cdots<i_q\leqslant i-1$ and $\ell+j-p-1\geqslant \ell\geqslant 0$. 

Suppose that $q<k$. Then $i\leqslant i_{q+1}<\cdots<i_k$, whence $i_k-i \geqslant i_k-i_{q+1} \geqslant k-(q+1) =k-q-1$, and so 
$-i_k\leqslant q-i-k+1$. On the other hand, since $i_k\leqslant n-\ell+k-2$, we get 
$\ell\leqslant n-i_k+k-2 \leqslant n+(q-i-k+1)+k-2=n+q-i-1$, 
and so 
$
\ell+j-p-1\leqslant n-i+j+(q-p)-2\leqslant n-i+j+(q-p)-1. 
$
Thus, $w'\in\overline{W}$. 

Now, suppose that $q=k$, i.e. $i_k< i\leqslant n$. 
If $\ell\leqslant n-i+k$, then $\ell+j-p-1\leqslant n-i+j+(k-p)-1$ and so $w'\in\overline{W}$. 
Therefore, suppose that $\ell> n-i+k$. 
As we also have $\ell\leqslant n-i_k+k-2 \leqslant n-i_k+k-1$, by Lemma \ref{r6}, we get 
$a_{i-1}\cdots a_{n-\ell+k} a_{i_k}\cdots a_{i_1}b_{n-1}^\ell \sim_{\bar R} a_{i_k}\cdots a_{i_1}b_{n-1}^\ell$. 
Then, 
$$
w\sim_{\bar R} e_i f_j a_{i-1}\cdots a_{n-\ell+k}  a_{i_k}\cdots a_{i_1}b_{n-1}^\ell, 
$$
with $1\leqslant i_1<\cdots <i_k \leqslant n-\ell+k-1< n-\ell+k \leqslant i-1$, 
and so, just like in (\ref{byasinpio-}), we have 
$$
w\sim_{\bar R} e_i f_j  a_{i-1}\cdots a_{n-\ell+k} a_{i_k}\cdots a_{i_{p+1}} a_1^{j-1}b_{n-1}^{\ell+j-p-1} = w'', 
$$
with $j\leqslant i_{p+1} < \cdots < i_k < n-\ell+k <\cdots <i-1 \leqslant i-1$. 
Since the length of the word $a_{i-1}\cdots a_{n-\ell+k} a_{i_k}\cdots a_{i_{p+1}}$ is $i-n+\ell-p$ and 
$\ell+j-p-1 = n-i+j+(i-n+\ell-p)-1$, we get $w''\in\overline{W}$, as required. 
\end{proof} 

\smallskip 

Recall that, for $n\geqslant3$,  $\{a_{1},\dots, a_{n-2},b_{n-1}, e_{n-1}, f_{2}\}$ is a generating set of minimum size of $\mathcal{PIO}_n$. 
Taking into account the equalities 
$a_{n-1}=b_{n-1}a_{1}$, 
$b_{i}=a_{i}b_{n-1}$, $1\leqslant i\leqslant n-2$,  
$e_{i}=b_{n-1}^{n-i-1}e_{n-1}a_{1}^{n-i-1}$, $1\leqslant i\leqslant n-2$, 
and $f_{j}=a_{1}^{j-2}f_{2}b_{n-1}^{j-2}$, $3\leqslant j\leqslant n$, 
by making these substitutions (\textit{Tietze transformations}), into the set of relations $\bar R$, 
we obtain a set $\tilde R$ of  relations on the alphabet $\{a_{1},\dots, a_{n-2},b_{n-1}, e_{n-1}, f_{2}\}$ such that 
$\langle a_{1},\dots, a_{n-2},b_{n-1}, e_{n-1}, f_{2} \mid \tilde R\rangle$ is also a presentation of $\mathcal{PIO}_n$.

\section*{Acknowledgements} 

The first author was supported by Scientific and Technological Research Council of Turkey (\textsc{tubitak}) under the Grant Number 123F463. 
The author thanks to \textsc{tubitak} for their supports. 

The second author was supported by national funds through the FCT-Funda\c c\~ao para a Ci\^encia e a Tecnologia, I.P., 
under the scope of the Center for Mathematics and Applications projects 
UIDB/00297/2020 (https://doi.org/10.54499/UIDB/00297/2020) and \textsc{UIDP}/00297/2020 (https://doi.org/10.54499/UIDP/00297/2020).

\section*{Declarations} 

The authors declare no conflicts of interest.

\bigskip 

{\sf\small 
\noindent{\sc Hayrullah Ay\i k},
\c{C}ukurova University,
Department of Mathematics,
Sar\i \c{c}am, Adana,
Turkey;
e-mail: hayik@cu.edu.tr  

\medskip 

\noindent{\sc V\'\i tor H. Fernandes},
Center for Mathematics and Applications (NOVA Math)
and Department of Mathematics,
Faculdade de Ci\^encias e Tecnologia,
Universidade Nova de Lisboa,
Monte da Caparica,
2829-516 Caparica,
Portugal;
e-mail: vhf@fct.unl.pt.

\medskip 

\noindent{\sc Emrah Korkmaz}, 
\c{C}ukurova University,
Department of Mathematics,
Sar\i \c{c}am, Adana,
Turkey;
e-mail: ekorkmaz@cu.edu.tr 
}

\end{document}